\documentclass[11pt,oneside]{amsart}
\usepackage[T1]{fontenc}
\usepackage{textcomp}

\newcommand{\nc}{\newcommand}

\nc{\Glr}{\mathrm{GL}_n(\mathbb{R})}   \nc{\Glc}{\mathrm{GL}_n(\mathbb{C})}

\nc{\euler}{{\rm e}}

\nc{\ip}{\langle \cdot , \cdot \rangle}
\nc{\ipd}{\langle \hspace{-0.5mm}\langle \cdot , \cdot \rangle\hspace{-0.5mm}\rangle}
\nc{\ippd}{( \hspace{-0.5mm} ( \cdot , \cdot ) \hspace{-0.5mm} )}
\nc{\ipp}{(      \cdot , \cdot       )}
\nc{\la}{\langle} \nc{\ra}{\rangle}

\newcommand{\ipa}[2]{\langle {#1} , {#2} \rangle}

\newcommand{\ipda}[2]{\langle \hspace{-0.5mm}\langle {#1} , {#2} \rangle\hspace{-0.5mm}\rangle}

\theoremstyle{plain}
\newtheorem{theorem}{Theorem}[section]
\newtheorem{proposition}[theorem]{Proposition}
\newtheorem{corollary}[theorem]{Corollary}
\newtheorem{lemma}[theorem]{Lemma}

\theoremstyle{definition}
\newtheorem{definition}[theorem]{Definition}

\theoremstyle{remark}
\newtheorem{remark}[theorem]{Remark}

\newtheorem{example}[theorem]{Example}

\numberwithin{equation}{section}

\usepackage[top=1in, left=1in, right=1in, bottom=1in]{geometry}

\title{Anti-K\"ahlerian geometry on Lie groups}

\author{Edison Alberto Fern\'{a}ndez-Culma$^\dagger$}
\thanks{$^\dagger$Fully supported by CONICET (Argentina).}

\author{Yamile Godoy$^\ddagger$}
\thanks{$^\ddagger$Partially supported by \textsc{Conicet} (PIP
	112-2011-01-00670), \textsc{Foncyt} (PICT 2010 cat 1 proyecto 1716) \textsc{Secyt Univ.\thinspace Nac.\thinspace C\'{o}rdoba}.}

\address{Current affiliation: \newline \indent CIEM-CONICET \& FaMAF-Universidad Nacional de C\'ordoba, \newline \indent Ciudad Universitaria, \newline \indent (5000) C\'ordoba, \newline \indent Argentina}
\email{\{efernandez, ygodoy\}@famaf.unc.edu.ar}
\subjclass[2010]{ }

\begin{document}

\maketitle

\begin{abstract}
Let $G$ be a Lie group of even dimension and let $(g,J)$ be a left invariant anti-K\"ahler structure on $G$. In this article we study anti-K\"{a}hler structures considering the distinguished cases where the  complex structure $J$ is abelian or bi-invariant. We find that if $G$ admits a left invariant anti-K\"{a}hler structure $(g,J)$ where $J$ is abelian then the Lie algebra of $G$ is unimodular and $(G,g)$ is a flat pseudo-Riemannian manifold. For the second case, we see that for any left invariant metric $g$ for which $J$ is an anti-isometry we obtain that the triple $(G, g, J)$ is an anti-K\"ahler manifold.

Besides, given a left invariant anti-Hermitian structure on $G$ we associate a covariant $3$-tensor $\theta$ on its Lie algebra and prove that such structure is anti-K\"{a}hler if and only if $\theta$ is a skew-symmetric and pure tensor. From this tensor we classify the real 4-dimensional Lie algebras for which the corresponding Lie group has a left invariant anti-K\"{a}hler structure and study the moduli spaces of such structures (up to group isomorphisms that preserve the anti-K\"{a}hler structures).
\end{abstract}


\section{Introduction}

Anti-Hermitian geometry can be considered as a counterpart of Hermitian geometry: an almost
anti-Hermitian manifold is a triple $(M,g,J)$, where $(M,g)$ is a pseudo-Riemannian manifold
and $J$ is an almost complex structure on $M$ such that $J$ is symmetric for $g$. The idea
of an almost anti-Hermitian manifold goes back at least as far as \cite[\S 2]{GribachevMekerovDjelepov},
where such a manifold is called a \textit{generalized $B$-manifold}. In the literature, other names are also used
for this class of manifolds:  \textit{Norden Manifolds} \cite{CastroHervellaGarciaRio} or \textit{almost complex manifolds with a Norden metric} \cite{GanchevBorisov}, in honour to the Russian mathematician Aleksandr P. Norden.

Since 1985, anti-Hermitian geometry has been very extensively studied, and continues to be a subject of intense interest
in complex geometry and mathematical physics (\cite{BorowiecFerrarisFrancavigliaVolovich, BorowiecFrancavigliaVolovich}).
Many contributions to the field have been made by the Bulgarian geometry school
(for instance, see \cite{GribachevMekerovDjelepov,GanchevBorisov,Manev1,Mekerov2, Mekerov3, Teofilova5}).

In \cite{GribachevMekerovDjelepov} it was given a division of almost anti-Hermitian geometry
in $8$ types of geometries by means of representation theory of $O(n,n) \cap GL(n,\mathbb{C}) \cong O(n,\mathbb{C})$.
By following such result, the fundamental class of almost anti-Hermitian
manifolds is the family of \textit{anti-K\"{a}hler manifolds} (also known as \textit{$B$-manifold} \cite[\S 2]{GribachevMekerovDjelepov}, \textit{K\"{a}hlerian manifold with a Norden metric} \cite{GanchevBorisov} or \textit{K\"{a}hler-Norden manifold} \cite{Sluka}), which are anti-Hermitian manifolds
with a parallel complex structure.

Other strong motivation to study anti-K\"{a}hler manifolds comes from
the work of the physicists-mathematicians Andrzej Borowiec, Mauro Francaviglia,
Marco Ferraris and Igor Volovich in \cite{BorowiecFerrarisFrancavigliaVolovich,BorowiecFrancavigliaVolovich},
where it is proved that there is a one-to-one correspondence between (Einstein) holomorphic-Riemannian manifolds and
(Einstein) anti-K\"{a}hler manifolds \cite[Theorem 2.2, Theorem 5.1]{BorowiecFrancavigliaVolovich}.
The curvature properties of anti-K\"{a}hler manifolds have been studied by Karina Olszak (formerly known as Karina S{\l}uka)
in \cite{Sluka,Sluka2,Olszakk} and by Arif Salimov and Murat \.{I}{\c{s}}can in \cite{IscanSalimov}.
Spinor geometry and geometric analysis on anti-K\"{a}hler manifolds have been considered by
Nedim De{\u{g}}irmenci and {\c{S}}enay Karapazar in \cite{DegirmenciKarapazar1,DegirmenciKarapazar2}.
Very rencently, Antonella Nannicini has studied the generalized geometry of anti-Hermitian manifolds in \cite{Nannicini} where she builds complex Lie algebroids over anti-K\"{a}hler manifolds.

This work is intended as an attempt to study anti-K\"{a}hler geometry on Lie groups
and to motivate new properties of anti-K\"{a}hler manifolds. In this paper, we focus on anti-K\"ahler structures on Lie groups in the left invariant setting.
In the complex geometry of Lie groups, we have two distinguished classes of left invariant complex structures, namely,
abelian and bi-invariant complex structures.  We study anti-K\"{a}hler structures with complex structures in each class.

In Section 3 we deal with the case when $J$ is a bi-invariant complex structure  on a Lie group $G$. In such case, we see that if $J$ is an anti-isometry of a left invariant metric $g$, the triple $(G, g, J)$ is an anti-K\"ahler manifold. Besides, we prove a sort of converse of \cite[Proposition 3.3]{Teofilova5} which states that a semisimple Lie group $G$ admitting a bi-invariant complex structure $J$ satisfies that $(G, g, J)$ is an anti-K\"ahler-Einstein manifold with non-vanishing cosmological constant, where $g$ is the bi-invariant metric on $G$ induced by the Killing form of the Lie algebra of $G$.

In Section 4 we study left invariant anti-K\"{a}hler structures on Lie groups whose complex structure is abelian. In this context we find that if a Lie group $G$ admits a left invariant anti-K\"{a}hler structure $(g,J)$, where $J$ is an abelian complex structure, then the Lie algebra of $G$ is unimodular and $(G,g)$ is a flat pseudo-Riemannian manifold.

In Section 5, given a left invariant anti-Hermitian structure on a Lie group $G$, we associate a covariant $3$-tensor $\theta$ on its Lie algebra $\mathfrak{g}$ and prove that such structure is anti-K\"{a}hler if and only if $\theta$ is a skew-symmetric and pure tensor. The tensor $\theta$
allows us to define a subclass of \textit{anti-K\"{a}hler Lie groups} formed by those one whose $3$-tensor $\theta$ vanishes.
Any four dimensional Lie group with left invariant anti-K\"{a}hler structure belongs to such family, and so, in Section 6, we are concerned with the classification of such Lie groups and prove that there are exactly two non-abelian Lie algebras in dimension $4$ (up to isomorphism) admitting anti-K\"{a}hler structures. The end of Section 6 is devoted to study how many left invariant anti-K\"{a}hler structures a four dimensional simply connected Lie group can admit (up to group isomorphisms that preserve the anti-K\"{a}hler structure).

\subsection*{Acknowledgements}
The authors wish to extend their sincerest appreciation and thanks to Isabel Dotti and Marcos Salvai
for their corrections, comments and constructive criticisms.

\section{Preliminaries}

We start this section by giving the definition of an almost anti-Hermitian manifold.

\begin{definition}[Almost anti-Hermitian Manifold]
  An \textit{almost anti-Hermitian manifold} is a triple $(M,g,J)$, where $M$ is a differentiable manifold of real dimension $2n$,
  $J$ is an almost complex structure on $M$ and
  $g$ is an \textit{anti-Hermitian} metric on $(M,J)$, that is
\begin{eqnarray}\label{norden}
  g(JX,JY) &=& -g(X,Y),\, \forall X,Y \in \mathfrak{X}(M),
\end{eqnarray}
or equivalently, $J$ is symmetric with respect to $g$.

  If additionally $J$ is integrable, then the triple $(M,g,J)$
  is called an \textit{anti-Hermitian manifold} or \textit{complex Norden manifold}.
\end{definition}

\begin{remark}
If $(M,g,J)$ is an almost anti-Hermitian manifold, it is straightforward to check that the signature of $g$ is $(n,n)$, i.e.\ $g$ is a neutral metric.
\end{remark}

\begin{remark}\label{baseOrtonormal}
The linear (algebra) model of an almost anti-Hermitian manifold is given by a vector space $V$ of real dimension $2n$ with a linear complex structure
$J$ on $V$ ($J^2=-\operatorname{Id}$) and an inner product $\ip$ on $V$ such that $J$ is an anti-isometry of $(V,\ip)$; which, from now on, we shall call it
an \textit{anti-Hermitian vector space}. The action of $J$ on $V$ induces
a complex vector space structure on $V$ defined by $(a+\sqrt{-1}b)\cdot v := av + bJv$, for all $v \in V$.
Let $\ipd: V \times V \rightarrow \mathbb{C}$ given by:
\begin{eqnarray}
\ipda{v}{w}=\ipa{v}{w}-\sqrt{-1}\ipa{Jv}{w}.
\end{eqnarray}
It is straightforward to show that $\ipd$ is a $\mathbb{C}$-symmetric inner product on the complex space $(V,\mathbb{C})$ and even more
$\ip$ is the real part of $\ipd$.

By considering an orthonormal basis for the complex inner product space $(V,\ipd)$ (see \cite[The Basis Theorem 2.46.]{Harvey}), say $\{X_1,\,\ldots,\,X_n\}$, it follows that $\{X_1,\,JX_1,\,\ldots,\,X_n,JX_n\}$ is an orthonormal basis for the real inner product space $(V,\ip)$ (see \cite[Theorem 1.2]{GribachevMekerovDjelepov}, where such basis is called \textit{orthonormal $J$-basis}). And conversely, the real part (and the imaginary part) of a $\mathbb{C}$-symmetric inner product on a complex vector space $V$,
together with the linear complex structure $J$ given by $Jv:=\sqrt{-1}\cdot v$, define an anti-Hermitian vector space structure on $V$.
\end{remark}

\begin{remark}\label{igualdadgrupos}
  Let $(V,\ip,J)$ be an anti-Hermitian vector space and let
  \begin{eqnarray*}
  G:=\{T:V \rightarrow V : TJT^{-1}=J \mbox { and } \ipa{Tv}{Tw} = \ipa{v}{w},\, \forall v,w \in V \};
  \end{eqnarray*}
  i.e.\ $G$ is the intersection of the group of isometries of $(V,\ip)$ with the group that preserves the linear complex structure $J$.
  Because of Remark \ref{baseOrtonormal}, we have $G \cong \mathrm{O}(n,n) \cap \mathrm{GL}(n,\mathbb{C})$.
  Now, let us consider the complex inner product space $(V,\ipd)$ associated to $(V,\ip,J)$, just as in Remark \ref{baseOrtonormal}, and
  let
  \begin{eqnarray*}
  \widehat{G}&:=&\{T:V \rightarrow V : \ipda{Tv}{Tw} = \ipda{v}w\},
  \end{eqnarray*}
  the isometry group of $(V,\ipd)$. It follows from \cite[The Basis Theorem 2.46.]{Harvey}
  that $\widehat{G} \cong \mathrm{O}(n,\mathbb{C})$.
  An easy computation shows that the groups $G$ and $\widehat{G}$ are equal (and consequently, $\mathrm{O}(n,n) \cap \mathrm{GL}(n,\mathbb{C}) \cong \mathrm{O}(n,\mathbb{C})$).
\end{remark}

\begin{definition}[Anti-K\"{a}hler manifold]
An \textit{Anti-K\"{a}hler manifold}  is an almost anti-Hermitian manifold $(M,g,J)$ such that $J$ is parallel with respect to the Levi-Civita connection of the pseudo-Riemannian manifold $(M,g)$.
\end{definition}

Let $(M,g,J)$ be an almost anti-Hermitian manifold. From now on, let us denote by $\nabla$ the Levi-Civita connection of $(M,g)$
and we denote by $(\nabla_{X} J)$ the covariant derivative of $J$ in the direction of the vector field $X$. We recall that $(\nabla_{X} J)Y = \nabla_{X} JY - J \nabla_{X}Y$.

\begin{remark}
  Note that an anti-K\"{a}hler manifold $(M,g,J)$ satisfies that $J$ is integrable; it follows from the well known relation
  between the Nijenhuis tensor $N$ and the covariant derivative of $J$ with respect to the Levi-Civita connection of $(M,g)$
  \begin{eqnarray}\label{N}
N(X,Y) & := & [JX, JY] - J[JX,Y] -J[X,JY] - [X,Y]\\
 \nonumber           & = & (\nabla_{JX} J)Y -J(\nabla_{X} J)Y  -(\nabla_{JY} J)X + J(\nabla_{Y} J)X,
  \end{eqnarray}
 for all $X,Y$ in $\mathfrak{X}(M)$.
\end{remark}

The following lemma can be considered as an analogue to a well-known result in Hermitian geometry and it is a direct consequence of
the identity $X(g(JY,Z)) = X(g(Y,JZ))$ for all $X,Y,Z \in \mathfrak{X}(M)$ (since $J$ is symmetric with respect to $g$).

\begin{lemma}{\cite[Lemma 2.1.]{GribachevMekerovDjelepov}}\label{nablaJsimetrico}
  Let $(M,g,J)$ be an almost anti-Hermitian manifold. Then $(\nabla_X J)$ is a symmetric operator with respect to the metric $g$, i.e.
  \begin{eqnarray}
    g((\nabla_XJ)Y,Z) = g(Y, (\nabla_X J) Z),\, \forall X,Y,Z \in \mathfrak{X}(M).
  \end{eqnarray}
\end{lemma}

\begin{proposition}\cite[Theorem 2.4.c]{GribachevMekerovDjelepov}\label{implicaJparalelo}
  Let $(M,g,J)$ be an almost anti-Hermitian manifold. Then, $(M,g,J)$ is an anti-K\"{a}hler manifold
  if and only if
  \begin{eqnarray}
    (\nabla_{JX}J) Y = \varepsilon J(\nabla_{X} J)Y,\, \forall X,Y \in \mathfrak{X}(M)
  \end{eqnarray}
  where $\varepsilon$ is a real constant.
\end{proposition}

\begin{proof}
Define $\alpha(X,Y,Z) = g ((\nabla_{X}J)Y,Z)$. By Lemma \ref{nablaJsimetrico}, $\alpha$ is a tensor which is symmetric  in the last two variables:
\begin{eqnarray}
\label{proofeq1} \alpha(X,Y,Z) &=& \alpha(X,Z,Y)
\end{eqnarray}

Under the hypothesis we have
\begin{eqnarray}
\nonumber  \alpha(JX,Y,Z) &=& g ((\nabla_{JX}J)Y,Z) \\
\nonumber                 &=& \varepsilon g (J(\nabla_{X}J)Y,Z) \\
\nonumber                 &=& \varepsilon g ((\nabla_{X}J)Y,JZ); \mbox{ since } J \mbox{ is symmetric for } g\\
\label{proofeq2}          &=& \varepsilon \alpha(X,Y,JZ).
\end{eqnarray}
On the other hand,
\begin{eqnarray}
\nonumber  \alpha(JX,Y,Z) &=& \varepsilon g (J(\nabla_{X}J)Y,Z)  \\
\nonumber                 &=& -\varepsilon g ((\nabla_{X}J)JY,Z); \mbox{ since } J \mbox{ anti-commute with } (\nabla_{X}J)\\
\label{proofeq3}          &=& -\varepsilon \alpha(X,JY,Z).
\end{eqnarray}
Combining these three relations, we have
\begin{eqnarray}
  \alpha(JX,Y,Z)          & \stackrel{(\ref{proofeq1})}{=} & \alpha(JX,Z,Y)\\
\nonumber                 & \stackrel{(\ref{proofeq3})}{=} & -\varepsilon\alpha(X,JZ,Y)\\
\nonumber                 & \stackrel{(\ref{proofeq1})}{=} & -\varepsilon\alpha(X,Y,JZ)\\
\nonumber                 & \stackrel{(\ref{proofeq2})}{=} & -\alpha(JX,Y,Z)
\end{eqnarray}
Therefore $\nabla J =0$.
\end{proof}

\begin{definition}[Twin metric]
Let $(M,g, J)$ be an almost anti-Hermitian manifold. The tensor defined by the formula $\widetilde{g}(X,Y) := g(JX,Y),\, \forall X,Y \in \mathfrak{X}(M)$
is symmetric because of equation (\ref{norden}); we have even more, that $(M,\widetilde{g},J)$ is an almost anti-Hermitian manifold. The metric $\widetilde{g}$
is called \textit{associated metric} (\cite{GanchevBorisov}), \textit{twin metric} or \textit{dual metric} (\cite{BorowiecFerrarisFrancavigliaVolovich}).
\end{definition}

\begin{remark}
  Let $(M,g,J)$ be an anti-K\"{a}hler manifold and $\widetilde{g}$ its twin metric. It is proved in \cite[Theorem 5]{IscanSalimov} that the Levi-Civita connection of
  the twin metric coincides with the Levi-Civita connection of $g$. In particular $(M,\widetilde{g},J)$ is also an anti-K\"{a}hler manifold. Using this fact, it is proved in \cite[Theorem 6]{IscanSalimov}
  that the Riemannian curvature tensor of $(M,g,J)$  is \textit{pure}, i.e. for smooth vector fields $X,Y,Z,W$,
  \begin{equation}
  R(JX,Y,Z,W) =     R(X,JY,Z,W) =     R(X,Y,JZ,W) =     R(X,Y,Z,JW).
  \end{equation}
  It had been proved previously in \cite[Equation (16)]{Sluka2}.
\end{remark}

\subsection{Left invariant geometric structures on Lie groups}

We now proceed to consider Lie groups endowed with left invariant geometric structures.
Let $G$ be a Lie group and let us denote by $\mathfrak{g}$ its Lie algebra, which is
the finite dimensional real vector space consisting of all smooth vector fields invariant
under left translations $L_{p}$, $p\in G$. If $g$ is a left invariant
pseudo-Riemannian metric on $G$, i.e.\,the left translations are isometries of $(G,g)$, then
$g$ is completely determined by the inner product $\ip$ on $\mathfrak{g}$ induced by $g$:
\begin{eqnarray*}
  \ipa{X}{Y} &= & g(X,Y),\, \forall X,Y \in \mathfrak{g}.
\end{eqnarray*}

Conversely, every inner product on $T_eG$ (here, $e$ is the identity element of $G$),
or equivalently an inner product on $\mathfrak{g}$,  defines a left invariant pseudo-Riemannian metric on $G$.
The Levi-Civita connection of $(G,g)$ is a left invariant affine connection, that is, if $X, Y \in \mathfrak{g}$
then $\nabla_{X} Y \in \mathfrak{g}$. Besides, since $g(U,V)$ is a constant function on $G$ for all $U,V$ in $\mathfrak{g}$, we have that
$\nabla_{X}$ satisfies
\begin{eqnarray}
\label{skewnabla}                   0 &=& g(\nabla_{X}Y , Z) + g(Y , \nabla_{X}Z),\,
\end{eqnarray}
and, from the Koszul formula,
\begin{eqnarray}
\label{koszul}  g(\nabla_{X}Y , Z) &=&\tfrac{1}{2}\left\{ g([X,Y],Z) - g([Y,Z],X) + g([Z,X],Y)\right\},
\end{eqnarray}
for $X,Y,Z$ in $\mathfrak{g}$.

An almost complex structure $J$ on a Lie group $G$ is said to be \textit{left invariant} if $(\operatorname{d}L_p)\circ J = J \circ(\operatorname{d}L_p)$ for all $p\in G$; equivalently for all $X \in \mathfrak{g}$, $J \circ X \in \mathfrak{g}$. Therefore, $J$ is completely determined by the linear complex structure $J_e : T_eG \rightarrow T_eG$. Conversely, every linear complex transformation on $T_eG$ determines
a left invariant almost complex structure $J$ on $G$, which is integrable if the Nijenhuis tensor $N$ given in (\ref{N}) vanishes
on $\mathfrak{g}$.

\begin{definition}[Abelian complex structure]
A left invariant almost complex structure $J$ on a Lie Group $G$ is called \textit{abelian} when it satisfies
\begin{eqnarray}\label{abeliancomplexdefinition}
  [JX,JY] = [X,Y],\, \forall X,Y \in \mathfrak{g}.
\end{eqnarray}
\end{definition}

\begin{remark}
  Note that an abelian complex structure $J$ on a Lie group $G$ is in fact integrable, hence $(G,J)$ is a complex manifold, but $(G,J)$ is not
  a complex Lie group (unless $\mathfrak{g}$ is a commutative Lie algebra).
\end{remark}

The notion of abelian complex structure has a important role in the complex geometry
of Lie groups. Such notion was introduced by Isabel Dotti, Roberto Miatello and Laura Barberis
in \cite{BarberisDottiMiatello} and since then, it has been extensively studied.

\begin{definition}[Bi-invariant complex structure]
A left invariant almost complex structure $J$ on a Lie group is called \textit{bi-invariant} if it satisfies
\begin{eqnarray}\label{biinvariant}
  [JX,Y] = J[X,Y] (=[X,JY]),\, \forall X,Y \in \mathfrak{g}.
\end{eqnarray}
\end{definition}

\begin{remark}
  Note that a bi-invariant complex structure $J$ on a Lie group $G$ is in fact integrable, and even more,
  $(G,J)$ is a complex Lie group.
\end{remark}

We are interested in studying anti-K\"{a}hler structures in the left invariant setting.
From now on, we say that $(g,J)$ is a \textit{left invariant almost anti-Hermitian structure} on a Lie group $G$
if $(G,g,J)$ is an almost anti-Hermitian manifold where $g$ and $J$ are left invariant geometric structures on $G$. In addition, if $(G,g,J)$ is an anti-K\"{a}hler manifold, we say that $(g,J)$ is a \textit{left invariant anti-K\"{a}hler structure} on $G$.
The following proposition provides sufficient conditions for a left invariant almost anti-Hermitian structure on a Lie group to be a left invariant anti-K\"{a}hler structure.

\begin{proposition}\label{nablaabelianabiinvariante}
Let $(g,J)$ be a left invariant almost anti-Hermitian structure on a Lie group $G$. If any of the following conditions are satisfied:
\begin{eqnarray}
\label{nablaimplicaabeliana}  \nabla_{JX}Y &=& -J\nabla_{X}Y,\,\,\, \forall X,Y \in \mathfrak{g}, \\
\label{nablaimplicabiinvariante}  \nabla_{JX}Y &=& J\nabla_{X}Y,\,\,\, \forall X,Y \in \mathfrak{g},
\end{eqnarray}
then $(G,g,J)$ is an anti-K\"{a}hler manifold, and even more, $J$ is an abelian complex structure if condition (\ref{nablaimplicaabeliana}) holds
and $J$ is a bi-invariant complex structure if condition (\ref{nablaimplicabiinvariante}) is satisfied.
\end{proposition}

\begin{proof}
  We begin by noting that any of the conditions implies that
  $$
  (\nabla_{JX}J)Y = \varepsilon J(\nabla_{X}J) Y,\,\,\, \forall X,Y \in \mathfrak{g};
  $$
with $\varepsilon =-1$ when the condition (\ref{nablaimplicaabeliana}) is satisfied, and $\varepsilon = 1$ in the other case.

Since $B(X,Y) = (\nabla_{JX}J) Y - \varepsilon J(\nabla_{X} J)Y$ is a $(2,1)$-tensor field on $G$
and $B$  vanishes for all $X, Y$ in $\mathfrak{g}$, we have that $B$  vanishes identically. From Proposition \ref{implicaJparalelo},
it follows that $(G,g,J)$ is an anti-K\"{a}hler manifold.

We now proceed with the proof of the last part of the proposition. If condition (\ref{nablaimplicaabeliana}) is satisfied,
let $X,Y$ in $\mathfrak{g}$:
\begin{eqnarray*}
  [JX,JY] &=& \nabla_{JX} JY - \nabla_{JY} JX\\
          &\stackrel{(\ref{nablaimplicaabeliana})}{=}&  -J \nabla_{X} J Y + J \nabla_{Y} J X\\
          &\stackrel{(\nabla J)\equiv 0 }{=}& \nabla_{X} Y - \nabla_{Y}X\\
          & = & [X,Y].
\end{eqnarray*}
By a similar argument, it is easy to check that condition (\ref{nablaimplicabiinvariante})
implies that $J$ is a bi-invariant complex structure.
\end{proof}

In the following two propositions we state and prove the converse of Proposition \ref{nablaabelianabiinvariante}.

\begin{proposition}\label{Jabelianasaleconnection}
  Let $(g,J)$ be a left invariant anti-K\"{a}hler structure on a Lie group $G$ such that $J$ is an abelian complex structure. Then
  $(G,g,J)$ satisfies the condition (\ref{nablaimplicaabeliana}), i.e.
  \begin{eqnarray*}
    \nabla_{JX}Y = -J\nabla_{X}Y,\, \,\,\forall X,Y \in \mathfrak{g}.
  \end{eqnarray*}
\end{proposition}

\begin{proof}
Let us consider the real inner product $\ip$ on $\mathfrak{g}$ induced by the metric $g$.
Define $\alpha(X,Y,Z) = \ipa{\nabla_{JX}Y + J\nabla_{X}Y}{Z}$, $\forall X,Y,Z$ in $\mathfrak{g}$.

From property (\ref{skewnabla}) and the fact that $J$ is parallel and symmetric for $\ip$, it follows that $\alpha$ is skew-symmetric in the last two variables
\begin{eqnarray}\label{proofabe1}
  \alpha(X,Y,Z)&=& -\alpha(X,Z,Y).
\end{eqnarray}
Since $J$ is an abelian complex structure, it is immediate to see that
\begin{eqnarray*}
\nabla_{JX}Y+J\nabla_{X}Y & = & \nabla_{JY}X + J\nabla_{Y}X
\end{eqnarray*}
and therefore $\alpha$ is symmetric in the first two variables:
\begin{eqnarray}\label{proofabe2}
  \alpha(X,Y,Z)&=& \alpha(Y,X,Z).
\end{eqnarray}
Combining (\ref{proofabe1}) with (\ref{proofabe2}) yields that $\alpha$ is symmetric in the first and last variables:
\begin{eqnarray}
\nonumber               \alpha(X,Y,Z) & \stackrel{(\ref{proofabe1})}{=} & -\alpha(X,Z,Y) \\
\nonumber                             & \stackrel{(\ref{proofabe2})}{=} & -\alpha(Z,X,Y) \\
\label{proofabe3}                     & \stackrel{(\ref{proofabe1})}{=} & -(-\alpha(Z,Y,X)).
\end{eqnarray}
Finally, the symmetry given by (\ref{proofabe3}) and (\ref{proofabe2}) implies that $\alpha$ is
a symmetric tensor on $\mathfrak{g}$, since the group $S_3$ is generated by the transpositions $(1\,2)$ and $(1\, 3)$.
In particular, $\alpha$ is symmetric in the last two variables:
\begin{eqnarray}
\nonumber               \alpha(X,Y,Z) & \stackrel{(\ref{proofabe1})}{=} & \alpha(Y,X,Z) \\
\nonumber                             & \stackrel{(\ref{proofabe3})}{=} & \alpha(Z,X,Y) \\
\label{proofabe4}                     & \stackrel{(\ref{proofabe1})}{=} & \alpha(X,Z,Y).
\end{eqnarray}
Comparing (\ref{proofabe4}) and (\ref{proofabe1}) we obtain that $\alpha$ vanishes identically, which establishes the proof.
\end{proof}

\begin{proposition}
  Let $(g,J)$ be a left invariant anti-K\"{a}hler structure on a Lie group $G$ such that $J$ is a bi-invariant complex structure.
  Then $(G,g,J)$ satisfies condition (\ref{nablaimplicabiinvariante}), i.e.
  \begin{eqnarray*}
    \nabla_{JX}Y = J\nabla_{X}Y,\,\,\, \forall X,Y \in \mathfrak{g}.
  \end{eqnarray*}
\end{proposition}

\begin{proof}
  The proof is straightforward and follows directly from the fact that the Levi-Civita is a torsion-free connection:
Since $[JX,Y]=J[X,Y]$, $\forall X,Y \in \mathfrak{g}$, we have $\nabla_{JX}Y - \nabla_{Y}JX = J(\nabla_{X}Y - \nabla_{Y}X)$.
Since $\nabla J \equiv 0$, we have $\nabla_{Y}JX = J\nabla_{Y}X$, and so $\nabla_{JX}Y = J\nabla_{X}Y$.
\end{proof}

\section{Anti-K\"{a}hler geometry on complex Lie groups}

It has already been proved in \cite{BorowiecFrancavigliaVolovich} that
\begin{proposition}\cite[Proposition 4.1]{BorowiecFrancavigliaVolovich}
  Every complex parallelizable manifold $M$ admits an anti-K\"{a}hler structure.
\end{proposition}
This proposition is a very contrasting result with the K\"ahler case, where the only compact complex parallelizable manifold
admitting K\"{a}hler structures are \textit{complex tori}; a well known result due to Hsien-Chung Wang (\cite[Corollary 2]{Wang}).

Other proof of the existence of an anti-K\"{a}hler structure on complex Lie group is given
by Marta Teofilova in \cite[Proposition 3.1]{Teofilova5}. The following proposition may be
considered as a generalization of the above mentioned result.

\begin{proposition}\label{JbiinvariantAntiKahler}
Let $(g,J)$ be a left invariant almost anti-Hermitian structure on a Lie group $G$
where $J$ is a bi-invariant complex structure on $G$. Then $(G,g,J)$ is
an anti-K\"{a}hler manifold.
\end{proposition}

\begin{proof}
Since $(\nabla J)(\cdot , \cdot)$ is a tensor field
on $G$, it is sufficient to verify that $\nabla_{X} J Y = J \nabla_{X} Y$ for all $X,Y$ in $\mathfrak{g}$.
For an arbitrary $Z \in \mathfrak{g}$ we have
\begin{eqnarray*}
2g(\nabla_{X}JY,Z) & \stackrel{(\ref{koszul})}{=}                 & g([X,JY],Z) - g([JY,Z],X) + g([Z,X],JY)\\
                   & \stackrel{(\ref{biinvariant})}{=}            & g(J[X,Y],Z) - g([Y,JZ],X) + g([Z,X],JY) \\
                   & \stackrel{(\ref{norden})}{=}  & g([X,Y],JZ) - g([Y,JZ],X) + g(J[Z,X],Y)\\
                   & \stackrel{(\ref{biinvariant})}{=}            & g([X,Y],JZ) - g([Y,JZ],X) + g([JZ,X],Y)\\
                   & \stackrel{(\ref{koszul})}{=}                 & 2g(\nabla_X Y,JZ)\\
                   & \stackrel{(\ref{norden})}{=}  & 2g(J\nabla_X Y,Z),
\end{eqnarray*}
and the proposition follows.
\end{proof}

\begin{remark}
Combining  Remark \ref{baseOrtonormal} and Proposition \ref{JbiinvariantAntiKahler} with well-known results
of representation of algebras and \textit{wild problems} (also known as \textit{hopeless problems}), we have that the classification of anti-K\"ahler
manifolds could be a wild problem.
\end{remark}

To finish this section, we want to study a result due to Marta Teofilova concerning almost anti-Hermitian structures on complex semisimple Lie groups
(see \cite[Proposition 3.3]{Teofilova5}). We improve slightly this result, by using the previous  proposition and well-known results
on the Killing form of semisimple Lie algebras, namely that the Killing form of a semisimple Lie algebra $\mathfrak{g}$ is an inner product
(Cartan's criterion) inducing a bi-invariant Einstein metric on a Lie group $G$ with Lie algebra $\mathfrak{g}$. We denote by $R$ and $\operatorname{Rc}$ the Riemannian curvature tensor and the Ricci tensor of a pseudo-Riemannianan manifold, respectively.

\begin{proposition}
 Let $G$ be a semisimple Lie group admitting a bi-invariant complex structure $J$.
 If $g$ is the bi-invariant metric on $G$ induced by the Killing form of $\mathfrak{g}$
 then $(G,g,J)$ is an anti-K\"{a}hler-Einstein manifold with non-vanishing cosmological constant.
\end{proposition}

About the converse of above proposition, we can prove:

\begin{proposition}
  Let $G$ be a Lie group admitting a left invariant anti-K\"{a}hler-Einstein structure $(g,J)$
  with non-vanishing cosmological constant and $g$ a bi-invariant metric. Then, $G$ is
  a semisimple Lie group and $J$ is a bi-invariant complex structure on  $G$.
\end{proposition}

\begin{proof}
  Since $g$ is a bi-invariant metric on $G$, we have $\operatorname{Rc}(X,Y) = \frac{1}{4}B(X,Y)$ for all $X,Y$ in $\mathfrak{g}$, where $B$ is the Killing form of $\mathfrak{g}$. And so, from the hypothesis that $g$ is an Einstein metric with non-vanishing cosmological constant, it follows
  that $\mathfrak{g}$ is a semisimple Lie algebra (Cartan's criterion).

  Next, we show that $J$ is a bi-invariant complex structure on $G$. Again, since $g$ is a bi-invariant metric on $G$,
  we have $R(X,Y)Z=-\frac{1}{4}[[X,Y],Z]$ for all $X,Y,Z$ in $\mathfrak{g}$. Since $(G,g,J)$ is
  an anti-K\"{a}hler manifold, it is easy to see that $R(JX,JY)=-R(X,Y)$ for all $X,Y$ in $\mathfrak{X}(M)$,
  because the \textit{symmetry by pairs} of the Riemannian curvature tensor of $(M,g)$ and $(\nabla J)\equiv 0$.
  This clearly forces that the complex structure $J$ is \textit{anti-abelian}; i.e.\ $[JX,JY]=-[X,Y]$ for all $X,Y$ in $\mathfrak{g}$,
  since $\mathfrak{g}$ is centerless. From the vanishing of the Nijenhuis tensor (see (\ref{N})) and the anti-abelian property of $J$, we have that $J$ is bi-invariant complex structure.
\end{proof}

\begin{remark}
There exist Lie groups admitting (non-flat) Ricci-flat anti-K\"{a}hler structures where the metric and the
complex structure are bi-invariant. Such Lie group is necessarily a solvable Lie group with vanishing Killing form.
\end{remark}

\section{Anti-K\"{a}hler geometry and abelian complex structures}

In this section, we study left invariant anti-K\"{a}hler structures with abelian complex structures. We begin with an elementary observation
about the Levi-Civita connection under the mentioned hypothesis.

\begin{proposition}\label{levicivitaabeliana}
Let $G$ be a Lie group admitting a left invariant anti-K\"{a}hler structure $(g,J)$
with $J$ an abelian complex structure on $G$. Then the Levi-Civita connection
of $(G,g)$ is completely determined just by the complex structure $J$ and the Lie algebra $\mathfrak{g}$:
\begin{eqnarray}
  \nabla_{X}Y = \tfrac{1}{2}\left( [X,Y] - J[X,JY] \right),\, \forall X,Y \in \mathfrak{g}
\end{eqnarray}
\end{proposition}

\begin{proof}
  Since $(\nabla J)\equiv 0$ and $\nabla$ is torsion-free, it is easy to see that
  $\forall X,Y$ in $\mathfrak{X}(G)$
\begin{eqnarray*}
  [JX,Y] - J[X,Y] & = & \nabla_{JX} Y - J\nabla_{X}Y.
\end{eqnarray*}
Combining this with Proposition \ref{Jabelianasaleconnection}, we have
\begin{eqnarray*}
           \nabla_{X}Y &=& \tfrac{1}{2}\left( [X,Y] + J[JX,Y]\right)\\
                       &\stackrel{(\ref{abeliancomplexdefinition})}{=}  & \tfrac{1}{2}\left( [X,Y] - J[X,JY]\right)
\end{eqnarray*}
for all $X,Y$ in $\mathfrak{g}$.
\end{proof}

\begin{corollary}
  Let $G$ be a Lie group and let $(g, J)$ be a left invariant anti-K\"{a}hler structure on $G$ where $J$ is an abelian complex structure on $G$. If $g$ is a bi-invariant metric on $G$ then the Lie algebra of $G$ is commutative.
\end{corollary}

The following proposition gives an important obstruction
to a Lie group admitting a left invariant anti-K\"{a}hler
structure $(g,J)$ with $J$ an abelian complex structure on $G$.

\begin{proposition}
  If $G$ is a $2n$-dimensional Lie group admitting a left invariant anti-K\"{a}hler
  structure $(g,J)$ with $J$ an abelian complex structure,
  then $\mathfrak{g}$ is a unimodular Lie algebra, i.e. for all $X$ in $\mathfrak{g}$
\begin{eqnarray*}
  \operatorname{Tr}(\operatorname{ad}_X) &=& 0.
\end{eqnarray*}
\end{proposition}

\begin{proof}
Combining the above proposition and condition (\ref{skewnabla}), we have for all $X,Y$ in $\mathfrak{g}$
\begin{eqnarray*}
  0 & = & 2g(\nabla_X Y , Y)\\
    & = &  g([X,Y] - J[X,JY] , Y)\\
    & = &  g([X,Y],Y) - g([X,JY],JY).
\end{eqnarray*}

Now we consider the inner product $\ip$ on $\mathfrak{g}$ induced by the left invariant metric $g$ on $G$
and let $\{Y_1,\ldots,Y_{2n}\}$ be an orthonormal basis of $(\mathfrak{g},\ip)$. By using this basis,
we have for all $X$ in $\mathfrak{g}$
\begin{eqnarray*}
  \operatorname{Tr}(\operatorname{ad}_X) & = & \sum_{i=1}^{2n} \ipa{Y_i}{Y_i}\ipa{[X,Y_i]}{Y_i}.
\end{eqnarray*}
On the other hand, by using the orthonormal basis $\{JY_1,\ldots,JY_{2n}\}$ of $(\mathfrak{g},\ip)$
we have
\begin{eqnarray*}
  \operatorname{Tr}(\operatorname{ad}_X) & = & \sum_{i=1}^{2n} \ipa{JY_i}{JY_i}\ipa{[X,JY_i]}{JY_i}\\
                                         &\stackrel{(\ref{norden})}{=}  &  \sum_{i=1}^{2n} -\ipa{Y_i}{Y_i}\ipa{[X,JY_i]}{JY_i}.
\end{eqnarray*}
By adding the last two expressions, we have for all  $X$ in $\mathfrak{g}$
\begin{eqnarray*}
  2 \operatorname{Tr}(\operatorname{ad}_X) & = & \sum_{i=1}^{2n} \ipa{Y_i}{Y_i}(\ipa{[X,Y_i]}{Y_i}-\ipa{[X,JY_i]}{JY_i}),
\end{eqnarray*}
but, it follows from the first expression that each summand in the last formula for $\operatorname{Tr}(\operatorname{ad}_X)$
is equal to zero, and so $\operatorname{Tr}(\operatorname{ad}_X)=0$.
\end{proof}

Under the hypothesis of the previous proposition, we have as a corollary that the abelian complex structure $J$ is symmetric for the Killing form.

\begin{corollary}
  If $G$ is a Lie group admitting a left invariant anti-K\"{a}hler
  structure $(g,J)$ with $J$ an abelian complex structure,
  then
  \begin{eqnarray*}
    B(JX,JY) &=& -B(X,Y),\, \forall X,Y \in \mathfrak{g},
  \end{eqnarray*}
  where $B$ is the Killing form of $\mathfrak{g}$.
\end{corollary}

\begin{proof}
First, it follows from the above proposition that $\operatorname{Tr}(\operatorname{ad}_{J[X,Y]}) = 0$
for all $X,Y$ in $\mathfrak{g}$. Since $\operatorname{ad}_{J[X,Y]} = -\operatorname{ad}_{[X,Y]}J = \operatorname{ad}_{Y}\operatorname{ad}_{X}J
-\operatorname{ad}_{X}\operatorname{ad}_{Y}J$, we have
\begin{eqnarray*}
  B(JX,Y) & = & B(Y,JX) \\
          & = & \operatorname{Tr}(\operatorname{ad}_{Y}\operatorname{ad}_{JX})\\
          & = & -\operatorname{Tr}(\operatorname{ad}_{Y}\operatorname{ad}_{X}J)\\
          & = & -\operatorname{Tr}(\operatorname{ad}_{X}\operatorname{ad}_{Y}J)\\
          & = & \operatorname{Tr}(\operatorname{ad}_{X}\operatorname{ad}_{JY})\\
          & = & B(X,JY),\, \forall X,Y \in \mathfrak{g},
\end{eqnarray*}
and the corollary follows.
\end{proof}

We state and prove the next proposition which we will use to prove the main theorem of this section (see Theorem \ref{abelianflat}).

\begin{proposition}
If $G$ is a Lie group admitting a left invariant anti-K\"{a}hler
  structure $(g,J)$ with $J$ an abelian complex structure,
  then for all $X,Y,Z$ in $\mathfrak{g}$
\begin{eqnarray*}
  \nabla_{X}\nabla_{Y}Z = \nabla_{Y} \nabla_{X}Z.
\end{eqnarray*}
\end{proposition}

\begin{proof}
  Let $X,Y,Z$ in $\mathfrak{g}$. From Proposition \ref{levicivitaabeliana}, we have
  \begin{eqnarray*}
    \nabla_X\nabla_Y Z  & = & \frac{1}{2}\nabla_X ([Y,Z]-J[Y,JZ])\\
                        & = & \frac{1}{4}( [X,[Y,Z]-J[Y,JZ]] - J[X,J([Y,Z]-J[Y,JZ])]\\
                        & = & \frac{1}{4}( [X,[Y,Z]] - [X,J[Y,JZ]] - J[X,J[Y,Z]] - J[X,[Y,JZ]] ),
  \end{eqnarray*}
and so $4(\nabla_X\nabla_Y Z - \nabla_Y\nabla_X Z )$ is equal to
\begin{eqnarray*}
 & = &  [X,[Y,Z]] - [X,J[Y,JZ]] - J[X,J[Y,Z]] - J[X,[Y,JZ]]\\
 &   & -[Y,[X,Z]] + [Y,J[X,JZ]] + J[Y,J[X,Z]] + J[Y,[X,JZ]]\\
 &\stackrel{(\ref{abeliancomplexdefinition})}{=}  & [X,[Y,Z]] - [JX,[JY,Z]] + J[JX,[JY,JZ]] - J[X,[Y,JZ]]\\
 &                                                &-[Y,[X,Z]] + [JY,[JX,Z]] - J[JY,[JX,JZ]] + J[Y,[X,JZ]]\\
 &\stackrel{\mbox{Jacobi}}{=}  & [[X,Y],Z] + [[JY,JX],Z] + J[[JX,JY],JZ] + J[[Y,X],JZ]\\
  &\stackrel{(\ref{abeliancomplexdefinition})}{=}  & [[X,Y],Z] + [[Y,X],Z]+ J[[X,Y],JZ]+J[[Y,X],JZ]\\
  &=&0,
\end{eqnarray*}
as desired.
\end{proof}

\begin{theorem}\label{abelianflat}
 Let $(g,J)$ be a left invariant anti-K\"{a}hler structure on a Lie group $G$ such that
 $J$ is an abelian complex structure. Then $(G,g)$ is a flat pseudo-Riemannian Lie group.
\end{theorem}

\begin{proof}
  It is sufficient to prove that $R(X,Y)Z =0$ for all $X,Y,Z$ in $\mathfrak{g}$.
  By the above proposition, we have $R(X,Y)Z = \nabla_{[X,Y]}Z$
  and therefore
  \begin{eqnarray*}
  R(JX,JY)Z = \nabla_{[JX,JY]}Z &=& R(X,Y)Z.
  \end{eqnarray*}
  While on the other hand, by virtue of the \textit{symmetry by pairs} of the Riemannian curvature tensor of $(M,g)$ and $(\nabla J)\equiv 0$,
  \begin{eqnarray*}
   R(JX,JY)Z &=& -R(X,Y)Z
  \end{eqnarray*}
and the proof is completed.
\end{proof}

As a consequence of the above theorem and Proposition \ref{levicivitaabeliana} we have the following obstruction: if a Lie group admits a
 left invariant anti-K\"{a}hler structure with abelian complex structure, then its Lie algebra $\mathfrak{g}$ together with the abelian complex structure must satisfy a distinguished $3$-degree polynomial identity.

\begin{corollary}
  Let $(g,J)$ be a left invariant anti-K\"{a}hler structure on a Lie group $G$ such that
 $J$ is an abelian complex structure, then for all $X,Y,Z$ in $\mathfrak{g}$
 \begin{eqnarray*}
   [J[X,Y],Z] &=& J[[X,Y],Z].
 \end{eqnarray*}

\end{corollary}

We want to give an example of a left invariant anti-K\"{a}hler structure $(g,J)$ on a Lie group with
$J$ an abelian complex structure. The $6$-dimensional Lie algebras that can be endowed with abelian complex structure
were classified in \cite{AndradaBarberisDotti1}. We focus on the nilpotent Lie algebra $\mathfrak{n}_{7}$
of \cite[Theorem 3.4]{AndradaBarberisDotti1} with its abelian complex structure $J_{-1}$.

\begin{example}
Consider a nilpotent Lie group $N$ with
Lie algebra
\begin{eqnarray*}
  \mathfrak{n}_{7}&:=& \left\{\begin{array}{l}
   [X_1, X_2] = X_4,\, [X_1, X_3] = X_5,\, [X_1, X_4] = X_6,\\
   {[X_2, X_3]} = X_6,\, [X_2, X_4] = -X_5.
  \end{array}\right.
\end{eqnarray*}
Such Lie group $N$ admits an abelian complex structure determined by the linear complex structure
$J:=J_{-1}$ on $\mathfrak{n}_{7}$ defined by
\begin{eqnarray*}
  JX_1=X_2,\,JX_3=-X_4,\, JX_5=-X_6.
\end{eqnarray*}
Consider the left invariant metric $g$ on $N$ such that
\begin{eqnarray*}
  \{X_1+X_5,\, X_1-X_5,\, X_2+X_6,\, X_2-X_6,\, X_3,\, X_4\}
\end{eqnarray*}
is an orthonormal frame field, where $\{X_1+X_5,X_2+X_6,X_3\}$ are spacelike vector fields
and $\{X_1-X_5,X_2-X_6,X_4\}$ are timelike vector fields. In the frame field $\{X_1,\ldots,X_6\}$,
all possible non-zero functions of the form $g(X_i,X_j)$ are
\begin{eqnarray*}
  g(X_1,X_5) \equiv \frac{1}{2},\, g(X_2,X_6) \equiv \frac{1}{2},\, g(X_3,X_3) \equiv 1 ,\, g(X_4,X_4) \equiv -1.
\end{eqnarray*}
It is a simple matter to check that
\begin{eqnarray*}
  \nabla_{X_1} X_1 = -\frac{1}{2}X_3,\,   \nabla_{X_1} X_2 = \frac{1}{2}X_4,\,   \nabla_{X_1} X_3 = X_5,\,   \nabla_{X_1} X_4 = X_6,\\
  \nabla_{X_2} X_1 = -\frac{1}{2}X_4,\,   \nabla_{X_2} X_2 = -\frac{1}{2}X_3,\,   \nabla_{X_2} X_3=X_6,\,   \nabla_{X_2} X_4 = -X_5,
\end{eqnarray*}
and therefore we can prove that $(N,g,J)$ is an anti-K\"{a}hler manifold.
\end{example}

\section{Some $3$-forms associated with left invariant anti-K\"{a}hler structures on Lie groups}

Let $G$ be a Lie group admitting a left invariant anti-K\"{a}hler structure $(g,J)$.
We begin by defining a family of bilinear maps on $\mathfrak{g}$ in the following way: let $\{a_1,\ldots,a_4\}$ be real
constants and $D:\mathfrak{g}\times \mathfrak{g} \rightarrow \mathfrak{g}$ the bilinear map on $\mathfrak{g}$ given by
\begin{eqnarray*}
  D(X,Y) & = & a_{1}\nabla_{X} Y + a_{2}\nabla_{JX}Y + a_{3}J\nabla_{X}Y + a_{4}J\nabla_{JX}Y.
\end{eqnarray*}
Consider the inner product $\ip$ on $\mathfrak{g}$ induced by the metric and let $\beta$ be the covariant $3$-tensor on $\mathfrak{g}$
defined by
\begin{eqnarray*}
  \delta(X,Y,Z) & = & \ipa{D(X,Y)}{Z},\, \forall X,Y,Z\in \mathfrak{g}.
\end{eqnarray*}
Note that $\delta$ is skew-symmetric in the last two arguments, because of property (\ref{skewnabla}) and the anti-K\"{a}hler hypothesis.
Therefore, the skew-symmetric part of $\delta$ is a multiple of
\begin{eqnarray}\label{thetadef}
  \theta(X,Y,Z) & = & \ipa{D(X,Y)}{Z} + \ipa{D(Y,Z)}{X} + \ipa{D(Z,X)}{Y}.
\end{eqnarray}
We want to highlight an important member in this family of skew-symmetric tensors which is defined from
the particular bilinear map
\begin{eqnarray}\label{thetadefB}
D(X,Y) = \nabla_{JX}Y + J\nabla_{X}Y.
\end{eqnarray}
In this case, we have that $\delta$ is a \textit{pure}
tensor on $\mathfrak{g}$ since $D(JX,Y)=D(X,JY)$, $D(JX,Y)=JD(X,Y)$ and $\ipa{J \cdot}{\cdot} = \ipa{\cdot}{J \cdot}$. So $\theta$ is a pure skew-symmetric $3$-tensor on $\mathfrak{g}$.

Now, we consider on $\mathfrak{g}$ the complex vector space structure induced by $J$ $(a+\sqrt{-1}\,b)\cdot X := aX + JbX$
and call $\widehat{\theta}(X,Y,Z) = \theta(X,Y,Z) - \sqrt{-1}\,\theta(JX,Y,Z)$. In this way, we have on $(\mathfrak{g},\mathbb{C})$
a complex skew-symmetric $3$-tensor and $\theta$ is its real part. Furthemore, $\theta$ has the following very nice expression: for all
$X,Y,Z$ in $\mathfrak{g}$
\begin{eqnarray}\label{thetadefcorchete}
              \theta(X,Y,Z) &=&  \ipa{[JX,Y]}{Z} + \ipa{[JY,Z]}{X} + \ipa{[JZ,X]}{Y}.
\end{eqnarray}

Note that the above $\theta$ vanishes identically, for instance, when $J$ is an abelian complex structure (see Proposition \ref{Jabelianasaleconnection})
or when the dimension of $G$ is $4$. If $(g,J)$ is a left invariant anti-K\"{a}hler structure on a Lie group $G$ with
$g$ and $J$ bi-invariant geometric structures, then $\theta$ is a multiple of $\ipa{[JX,Y]}{Z}$.

Here is an important property of the $3$-tensor $\theta$, left invariant anti-K\"{a}hler
structures on Lie groups are determined by the skew-symmetry and pureness of its 3-tensor $\theta$:

\begin{theorem}
  Let $(g,J)$ be a left invariant almost anti-Hermitian structure on a Lie group $G$
  and let $\theta$ be the associated $3$-tensor as defined in (\ref{thetadef}) from (\ref{thetadefB}).
  Then, $(G,g,J)$ is an anti-K\"{a}hler manifold if and only if $\theta$ is skew-symmetric and pure on $\mathfrak{g}$; equivalently,
   \begin{eqnarray}
\label{proofthetaimplicaanti1}\theta(X,Y,Z) &=& - \theta(X,Z,Y) \\
  \nonumber &\mbox{ and}\\
\label{proofthetaimplicaanti2}\theta(JX,Y,Z) &=& \theta(X,JY,Z)
   \end{eqnarray}
 for all $X,Y,Z$ in $\mathfrak{g}$.
 \end{theorem}

\begin{proof}
  As we need to show that $\nabla J \equiv 0$, let us consider $\alpha(X,Y,Z) = \ipa{(\nabla_{X}J)Y}{Z}$ for all $X,Y,Z$ in $\mathfrak{g}$
  and we will show that $\alpha$ vanishes identically on $\mathfrak{g}$.

  From Lemma \ref{nablaJsimetrico}, we have that $\alpha$ is symmetric in the last two arguments:
  \begin{eqnarray*}
    \alpha(X,Y,Z) & = & \alpha(X,Z,Y),\, \forall X,Y,Z \in \mathfrak{g}.
  \end{eqnarray*}
  We also have the following equality:
  \begin{eqnarray}\label{prooftheta2}
    \alpha(X,Y,JZ) & = & -\alpha(X,JY,Z),\, \forall X,Y,Z \in \mathfrak{g}.
  \end{eqnarray}
From (\ref{proofthetaimplicaanti1}), it follows $\theta(X,Y,Z) + \theta(X,Z,Y) = 0$ for all $X,Y,Z$ in $\mathfrak{g}$.
  Now $\theta(X,Z,Y)$ is also equal to
  \begin{eqnarray*}
\ipa{Z}{-\nabla_{JX}Y-\nabla_{X}JY} + \ipa{Y}{-\nabla_{JZ}X-\nabla_{Z}JX} + \ipa{X}{-\nabla_{JY}Z-\nabla_{Y}JZ},
  \end{eqnarray*}
  hence
  \begin{eqnarray*}
    0 &=& \theta(X,Y,Z) + \theta(X,Z,Y)\\
      &=&  \ipa{\nabla_{JX}Y+J\nabla_{X}Y}{Z} + \ipa{\nabla_{JY}Z+J\nabla_{Y}Z}{X} + \ipa{\nabla_{JZ}X+J\nabla_{Z}X}{Y}\\
      & &  -\ipa{\nabla_{JX}Y+\nabla_{X}JY}{Z} - \ipa{\nabla_{JY}Z+\nabla_{Y}JZ}{X} - \ipa{\nabla_{JZ}X+\nabla_{Z}JX}{Y}\\
      &=&  -\ipa{\nabla_{X}JY -J\nabla_{X}Y }{Z}-\ipa{\nabla_{Y}JZ -J\nabla_{Y}Z }{X}-\ipa{\nabla_{Z}JX -J\nabla_{Z}X }{Y}\\
      &=&  -\alpha(X,Y,Z)-\alpha(Y,Z,X)-\alpha(Z,X,Y).
  \end{eqnarray*}
That is,  for all $X,Y,Z$ in $\mathfrak{g}$
\begin{eqnarray}\label{prooftheta3}
\alpha(X,Y,Z)+\alpha(Y,Z,X)+\alpha(Z,X,Y) &=& 0.
\end{eqnarray}
From (\ref{proofthetaimplicaanti2}), we have $\theta(X,JY,Z)- \theta(JX,Y,Z) = 0$.
Since $\theta(JX,Y,Z)$ is also equal to
\begin{eqnarray*}
  \ipa{-\nabla_{X}Y+J\nabla_{JX}Y}{Z}+\ipa{J\nabla_{JY}Z-\nabla_{Y}Z}{X}+\ipa{\nabla_{JZ}JX+J\nabla_{Z}JX}{Y},
\end{eqnarray*}
it follows
\begin{eqnarray*}
  0 &=& \theta(X,JY,Z)- \theta(JX,Y,Z) \\
    &=& \ipa{\nabla_{JX}JY+J\nabla_{X}JY}{Z}+\ipa{-\nabla_{Y}Z+J\nabla_{JY}Z}{X}+\ipa{J\nabla_{JZ}X-\nabla_{Z}X}{Y}\\
    & & +\ipa{\nabla_{X}Y-J\nabla_{JX}Y}{Z}+\ipa{-J\nabla_{JY}Z+\nabla_{Y}Z}{X}-\ipa{\nabla_{JZ}JX+J\nabla_{Z}JX}{Y}\\
    &=& \ipa{\nabla_{JX}JY -J\nabla_{JX}Y}{Z} + \ipa{\nabla_{X}Y+J\nabla_{X}JY}{Z}\\
    & & +\ipa{J\nabla_{JZ}X -\nabla_{JZ}JX }{Y} - \ipa{\nabla_{Z}X+J\nabla_{Z}JX}{Y}\\
    &=& \alpha(JX,Y,Z)-\alpha(X,JY,Z)-\alpha(JZ,X,Y)+\alpha(Z,JX,Y).
\end{eqnarray*}
Or even better, by adding $\alpha(Y,Z,JX)$ to both sides of the preceding identity and by using equality (\ref{prooftheta3}), we have
\begin{eqnarray*}
  \alpha(Y,Z,JX)  &=& -\alpha(X,JY,Z)-\alpha(JZ,X,Y).
\end{eqnarray*}
Now, by (\ref{prooftheta2}) we have that the above equality can be written as
\begin{eqnarray*}
  -\alpha(Y,JZ,X) & = & \alpha(X,Y,JZ)-\alpha(JZ,X,Y),
\end{eqnarray*}
and since, by (\ref{prooftheta3}),    $\alpha(X,Y,JZ)=-\alpha(Y,JZ,X)-\alpha(JZ,X,Y)$, finally we have that
\begin{eqnarray*}
	-\alpha(Y,JZ,X) & = & -\alpha(Y,JZ,X)-\alpha(JZ,X,Y)-\alpha(JZ,X,Y)
\end{eqnarray*}
So,  $\alpha(JZ,X,Y) = 0$ for all $X,Y,Z$ in $\mathfrak{g}$, which establishes the proof.
\end{proof}

\begin{corollary}
Let $(g,J)$ be a left invariant almost anti-Hermitian structure on a Lie group $G$
  such that the associated $3$-tensor $\theta$ vanishes identically on $\mathfrak{g}$.
Then $(G,g,J)$ is an anti-K\"{a}hler manifold.
\end{corollary}

Let us mention an important consequence of the theorem, which is the key to obtain the main result
of the following section:

\begin{corollary}\label{thetadim4}
  Let $(g,J)$ be a left invariant almost anti-Hermitian structure on a $4$-dimensional Lie group $G$.
  Then, $(G,g,J)$ is an anti-K\"{a}hler manifold if and only if the associated $3$-tensor $\theta$ as defined
  in (\ref{thetadef}) vanishes identically on $\mathfrak{g}$.
\end{corollary}

\section{Left invariant anti-K\"{a}hler structures on four dimensional Lie groups }

\begin{theorem}\label{maintheorem}
  Let $G$ be a $4$-dimensional Lie group. Then, $G$ admits a left invariant anti-K\"{a}hler
  structure if and only if its Lie algebra $\mathfrak{g}=\operatorname{Lie}(G)$ is
  abelian or is isomorphic to
  \begin{eqnarray*}
  \mathfrak{r}_{-1,-1}&=&\left\{[{e_1},{e_2}]={e_2},[{e_1},{e_3}]=-{e_3},[{e_1},{e_4}]=-{e_4},\right. \\
   & \text{or} & \\
  \mathfrak{aff}(\mathbb{C})_{\mathbb{R}} &=& \left\{[{e_1},{e_3}]={e_3},[{e_1},{e_4}]={e_4},[{e_2},{e_3}]={e_4},[{e_2},{e_4}]=-{e_3}.\right.
  \end{eqnarray*}
\end{theorem}

\begin{proof}
Let $G$ be a real Lie group of dimension four which admits a left invariant  anti-K\"ahler structure $(g,J)$.
As before, let us denote by $\ip$ the inner product on $\mathfrak{g}$ induced by the left invariant metric $g$
on $G$.

From Remark \ref{baseOrtonormal} there exists an orthonormal basis of the Lie algebra $\mathfrak{g}$ of $G$ of the form $\mathcal{B}=\{X,JX,Y, JY\}$,
 where $X$ and $Y$ are spacelike, and so, we have
$$
\left[ g\right] _{\mathcal{B}}=\text{diag}\,(1,-1,1,-1),
$$
and
$$
\left[ J\right] _{\mathcal{B}}=\text{diag}\,(j,j),
$$
where $j= \left[ \begin{array}{c c}
0 & -1\\
1 & 0
\end{array}\right] $.

From Corollary \ref{thetadim4} we have that the $3$-form $\theta$ vanishes on $\mathfrak{g}$. In particular, we have
$\theta(U,V,JV)=0$ for all $U,V$ in $\mathfrak{g}$; equivalently
\begin{eqnarray*}
\ipa{[V,U]}{V}&=&-\ipa{[V,JU]}{JV}.
\end{eqnarray*}
So the following equalities for the elements of $\mathcal{B}$ hold:
\begin{equation}\label{theta}
\begin{array}{lcllcl}
\ipa{[X,Y]}{X}&=&-\ipa{[X,JY]}{JX},   &   \ipa{[X,Y]}{JX}&=&\ipa{[X,JY]}{X}, \\
\ipa{[X,Y]}{Y}&=&-\ipa{[JX,Y]}{JY},    &  \ipa{[X,Y]}{JY}&=&\ipa{[JX,Y]}{Y},\\
\ipa{[X,JY]}{Y}&=&-\ipa{[JX,JY]}{JY},  &  \ipa{[X,JY]}{JY}&=&\ipa{[JX,JY]}{Y},\\
\ipa{[JX,Y]}{X}&=&-\ipa{[JX,JY]}{JX}, &   \ipa{[JX,Y]}{JX}&=&\ipa{[JX,JY]}{X}, \\
\ipa{[X,JX]}{X}&=&0,              &   \ipa{[X,JX]}{JX}&=&0,\\
\ipa{[Y,JY]}{Y}&=&0,              &   \ipa{[Y,JY]}{JY}&=&0.
\end{array}
\end{equation}
Hence, by (\ref{theta}) the Lie bracket of elements of $\mathcal{B}$ satisfies the following equations
\begin{equation}\label{coef}
\begin{array}{rcl}
\left[ X, JX\right] &=& a\,Y + b\,JY,\\
\left[ X, Y\right] &=& t_{1}\,X + t_{2}\,JX + t_{3}\,Y + t_{4}\,JY ,\\
\left[ X, JY\right] &=& -t_{2}\,X + t_{1}\,JX + t_{5}\,Y + t_{6}\,JY ,\\
\left[ JX, Y\right] &=& t_{7}\,X + t_{8}\,JX - t_{4}\,Y + t_{3}\,JY ,\\
\left[ JX, JY\right] &=& -t_{8}\,X + t_{7}\,JX - t_{6}\,Y + t_{5}\,JY,\\
\left[ Y, JY\right] &=& c\,X + d\,JX,
\end{array}
\end{equation}
where $a, b, c, d$ and $t_{i}$, for $i=1, \cdots, 8$,  are real numbers.

Besides, by using again the 3-form $\theta$, since $\theta(U,U,V)=0$ for all $U,V \in \mathfrak{g}$, we obtain
\begin{equation}\label{abcd}
\begin{array}{ccccccc}
a&=&t_{2}+t_{7},& &b&=&t_{8}-t_{1},\\
c&=&-(t_{4}+t_{5}),& & d&=&t_{3}-t_{6}.
\end{array}
\end{equation}

Now, computing all the Jacobi equations involving the elements of $\mathcal{B}$ and considering the equalities given in (\ref{coef}) we obtain the following equations
\begin{equation}\label{paralell}
\begin{array}{rcl}
- (t_{8}+t_{1}) \left[ X, JX\right] - b \left[ Y, JY\right] + t_{3}\Delta(X,Y) + t_{4}\Delta(X,JY)&=&0, \\
(t_{2}-t_{7}) \left[ X, JX\right] + a \left[ Y, JY\right] + t_{5}\Delta(X,Y) + t_{6}\Delta(X,JY)&=&0, \\
- d \left[ X, JX\right] + (t_{3}+t_{6}) \left[ Y, JY\right] - t_{1}\Delta(X,Y) + t_{2}\Delta(X,JY)&=&0, \\
- c \left[ X, JX\right] + (t_{4}-t_{5})\left[ Y, JY\right] + t_{7}\Delta(X,Y) - t_{8}\Delta(X,JY)&=&0,
\end{array}
\end{equation}
where $\Delta(U,V)=\left[ JU, V\right]-\left[ U, JV\right]$, for all $U,V \in \mathfrak{g}$.

If $ \{ \left[ X, JX\right],  \left[ Y, JY\right],\Delta(X,Y), \Delta(X,JY) \} $ is a linearly independent set, then we have that all the constants are zero and so $\mathfrak{g}$ is the four dimensional abelian Lie algebra.

From here on, we assume that this set is linearly dependent. By (\ref{coef}) and (\ref{abcd}), we see that
\begin{equation}\label{deltacoef}
\begin{array}{ccc}
\Delta(X,Y)&=& a X + bJX + cY +d JY, \\
\Delta(X,JY)&=& -b X + a JX + dY -c JY.
\end{array}
\end{equation}
Next, we study the set of the 4-tuples $(\lambda_{1}, \lambda_{2},\lambda_{3}, \lambda_{4})$ satisfying
$$
\lambda_{1}[X,JX] + \lambda_{2}[Y,JY] + \lambda_{3}\Delta(X,Y) + \lambda_{4}\Delta(X,JY)=0.
$$
In terms of $\mathcal{B}$, using (\ref{coef}) and (\ref{deltacoef}), this set coincides with the set of solutions of the homogeneous linear system $Ax=0$, where
$$
A=\left[\begin{array}{cccc}
0 & c &a &-b \\
0 & d &b &a \\
a & 0 &c &d \\
b & 0 &d &-c  \end{array} \right] .
$$
We have two cases to analize:
\smallskip

\noindent \textbf{Case 1.}
If at least one of the coefficients of $A$ is non zero, we have that the Lie algebra is
\begin{eqnarray}\label{caso1a}
\mu_{a,b,\varepsilon}&=&
\left\{
\begin{array}{lcl}
\left[ X, JX\right] &=& a\,Y + b\,JY, \\
\left[ X, Y\right] &=& a\,JX -\varepsilon b\,JY,\\
\left[ X, JY\right] &=& -a\,X + \varepsilon a\,JY ,\\
\left[ JX, Y\right] &=& b\,JX + \varepsilon b\,Y ,\\
\left[ JX, JY\right] &=& -b\,X - \varepsilon a\,Y,\\
\left[ Y, JY\right] &=& \varepsilon b\,X -\varepsilon a\,JX
\end{array}
\right.
\end{eqnarray}
with $\varepsilon^2 =1 $ and $a,b \in \mathbb{R}$ ($a,b\neq0$).

To prove this assertion, first we call
$$
v_1=\left[\begin{array}{c}
-t_{1}-t_{8}\\
t_{1} - t_{8}\\
t_{3}\\
t_{4}\end{array} \right], \,
v_2=\left[ \begin{array}{c}
t_{2}-t_{7}\\
t_{2} + t_{7}\\
t_{5}\\
t_{6}\end{array} \right], \,
v_3=\left[ \begin{array}{c}
t_{6}-t_{3}\\
t_{3} + t_{6}\\
-t_{1}\\
t_{2}\end{array} \right] \,\, \text{and} \,\,
v_4=\left[ \begin{array}{c}
t_{4}+t_{5}\\
t_{4} - t_{5}\\
t_{7}\\
-t_{8}\end{array} \right].
$$
By the assumption that at least one of the coefficients of $A$ is different from zero, the rank of $A$ is three. By (\ref{paralell}), $v_{1},\, v_{2}, \,v_{3}$ and $v_{4}$ are in the kernel of $A$, which has dimension one. So, these vectors are parallel and not simultaneously zero. We will suppose that $v_{1}$ is nonzero (the remaining cases give the same conditions on the coefficients $a, b, c, d$ and $t_{i}$). So, there exist real numbers $\alpha, \, \beta$ and  $\varepsilon$ such that
$$
v_{2}=\alpha v_{1},\,\,v_{3}=\beta v_{1},\, \, v_{4}= \varepsilon v_{1}.
$$
In particular, using the equality $v_{3}=\beta\, v_{1}$, by the equations we have
$$
\begin{cases}
\begin{array}{lcl}
t_{6}-t_{3}&=&\beta\,(-t_{1}-t_{8}), \\
t_{3} + t_{6}&=&\beta\, (t_{1} - t_{8}), \\
-t_{1}&=&\beta \, t_{3}
\end{array}
\end{cases}
$$
and so, we obtain that $t_{3}(1+\beta^2)=0$. Therefore, $t_{3}=0$ and in consequence $t_{1}=t_{5}=t_{7}=0$, since
$-t_{1}=\beta \, t_{3}, \, t_{5}=\alpha \, t_{3} \, \,\text{and}\,\,
t_{7}=\varepsilon \, t_{3}.
$
Now, by (\ref{abcd}) and the above computations we have
\begin{equation*}
\begin{array}{lclclcl}
a&=&t_{2}, & &b&=&t_{8}\\
c&=&-t_{4} & &d&=&-t_{6},
\end{array}
\end{equation*}
and using that $v_{4}=\varepsilon\, v_{1}$ we obtain the following two equalities
$$
\begin{cases}
\begin{array}{ccc}
c&=&\varepsilon b, \\
b&=&\varepsilon c.
\end{array}
\end{cases}
$$
So, $(\varepsilon^2 -1)c=0$. That is, $c=0$ or $\varepsilon^2 = 1$. But, if $c=0$ we have that all the coefficients are zero (since we have the equalities $d=\alpha\, c, \, a=-\beta\, c$ and $b=\varepsilon \, c$), which contradicts the hypothesis. Hence, $\varepsilon^2 = 1$ and this gives that $c = \varepsilon b$ and $d= -\varepsilon a$ (the last one holds since  $a=-\beta b= -\varepsilon \beta c$ and $\beta c=d$).

Finally, in this case we have that the family of Lie algebras  $\mu_{a,b,\varepsilon}$ that was obtained is isomorphic to the Lie algebra $\mathfrak{r}_{-1,-1}$. If we consider the change of basis given by the  matrix
$$
\phi=
 \left[ \begin {array}{cccc}
-\varepsilon a&-\varepsilon b&b&-a\\
\noalign{\medskip}
\varepsilon b&-\varepsilon a&a&b\\
\noalign{\medskip}
0&-\varepsilon&-1&0\\
\noalign{\medskip}
\varepsilon&0&0&-1
\end {array} \right]
$$
whose inverse is given by
$$
\phi^{-1}=\frac{1}{2(a^2+b^2)}
 \left[\begin {array}{cccc}
-\varepsilon a&\varepsilon b&0& \varepsilon\left( {a}^{2}+{b}^{2} \right)\\
\noalign{\medskip}
-\varepsilon b&-\varepsilon a&-\varepsilon \left( {a}^{2}+{b}^{2} \right)&0\\
\noalign{\medskip}
b&a&-({a}^{2}+{b}^{2})&0\\
\noalign{\medskip}
-a&b&0&-({a}^{2}+{b}^{2})
\end{array} \right],
$$
we have that $\phi$ is an isomorphism between the Lie algebra $\mu_{a,b,\varepsilon}$ and the Lie algebra $\mathfrak{r}_{-1,-1}$.

\noindent \textbf{Case 2.}
On the other hand, if all the coefficients of $A$ are zero, we have that $\mathfrak{g}$ is the underlying real Lie algebra
 of the $2$-dimensional complex Lie algebra $\mathfrak{aff}(\mathbb{C})$. Indeed, $a=b=c=d=0$ implies that $t_{5}=-t_{4}$, $t_{6}=t_{3}$, $t_{7}=-t_{2}$ and $t_{8}=t_{1}$, and hence
\begin{equation}\label{caso2}
\mu\,_{t_{1},t_{2},t_{3},t_{4}}=
\begin{cases}
\begin{array}{rcl}
\left[ X, JX\right] &=& 0,\\
\left[ X, Y\right] &=& t_{1}\,X + t_{2}\,JX + t_{3}\,Y + t_{4}\,JY ,\\
\left[ X, JY\right] &=& -t_{2}\,X + t_{1}\,JX - t_{4}\,Y + t_{3}\,JY ,\\
\left[ JX, Y\right] &=& -t_{2}\,X + t_{1}\,JX - t_{4}\,Y + t_{3}\,JY ,\\
\left[ JX, JY\right] &=& -t_{1}\,X - t_{2}\,JX - t_{3}\,Y - t_{4}\,JY,\\
\left[ Y, JY\right] &=& 0.
\end{array}
\end{cases}
\end{equation}

It is easy to check that $ J\left[ U, V\right]=\left[JU, V\right]$, for all $U,V \in \mathcal{B}$, and so, $J$ is a bi-invariant
complex structure the on Lie algebra $\mu\,_{t_{1},t_{2},t_{3},t_{4}}$.

Finally, the change of basis given by the matrix
$$
\phi=
\left[ \begin {array}{cccc}
t_{3}&-t_{4}&-t_{1}&t_{2}\\
\noalign{\medskip}t_{4}&t_{3}&-t_{2}&-t_{1}\\
\noalign{\medskip}-t_{1}&-t_{2}&-t_{3}&-t_{4}\\
\noalign{\medskip}t_{2}&-t_{1}&t_{4}&-t_{3}
\end {array} \right]
$$
whose inverse is
$$
\phi^{-1}=\frac{1}{{t_{1}}^{2}+{t_{2}}^{2}+{t_{3}}^{2}+{t_{4}}^{2}}\,\phi^{t}
$$
gives an isomorphism between $\mu\,_{t_{1},t_{2},t_{3},t_{4}}$ and
$$
\mathfrak{aff}(\mathbb{C})_{\mathbb{R}}=
\left\{[{e_1},{e_3}]={e_3},[{e_1},{e_4}]={e_4},[{e_2},{e_3}]={e_4},[{e_2},{e_4}]=-{e_3}.\right.
$$
\end{proof}

Now, we will  investigate a more delicate problem. We are interested in determining how many left invariant K\"{a}hler
structures  the Lie groups given in the Theorem \ref{maintheorem} have.
Rather than discussing this in full generality, let us assume that the Lie groups are simply connected
and introduce the notion of \textit{equivalence} of anti-K\"{a}hler structures

\begin{definition}
Let $(g_1,J_1)$ and $(g_2,J_2)$ be left invariant anti-K\"{a}hler structures on simply connected
Lie groups $G_1$ and $G_2$, respectively. We say that $(g_1,J_1)$ is equivalent to $(g_2,J_2)$
if there exists an isomorphism $\Psi:G_1 \rightarrow G_2$ such that is an isometry between
the pseudo-Riemannian manifolds $(G_1,g_1)$ and $(G_2,g_2)$, and
$(\operatorname{d}\Psi)\circ J_1 = J_2 \circ(\operatorname{d}\Psi)$.
\end{definition}

\begin{remark}
  Note that two anti-K\"{a}hler structures $(g_1,J_1)$ and $(g_2,J_2)$ are equivalent
  if and only if there exists a Lie algebra isomorphism $\psi:\mathfrak{g}_1 \rightarrow \mathfrak{g}_2$
  such that $\psi$ is a (linear) isometry between $(\mathfrak{g}_1,\ip_1)$ and $(\mathfrak{g}_2,\ip_2)$,
  and $\psi \circ J_1 = J_2 \circ \psi$. Here, $\ip_i$ is the respective inner product on $\mathfrak{g}_i$ induced by the left invariant
  metric $g_i$.
\end{remark}

\begin{remark}
  The proof of Theorem \ref{maintheorem} gives more information, namely that each Lie algebra in (\ref{caso1a}) and (\ref{caso2}) represents
  a left invariant anti-K\"{a}hler structure on the respective Lie group. Our equivalence
  problem reduces to know when $\mu_{a,b,\varepsilon}$ and $\mu_{c,d,\varepsilon}$ are isomorphic via a linear map $\varphi$ that preserves the inner product $\ip$ and the complex structure given in the beginning of the section ($\varphi \in  O(2,2) \cap GL(2,\mathbb{C})$). The same reasoning for two Lie
  algebras in the family given in (\ref{caso2}).
\end{remark}

\begin{theorem}
  Let $G$ be a simply connected Lie group of dimension $4$ admitting a left invariant anti-K\"{a}hler structure. If the Lie algebra  $\mathfrak{g}$ of $G$ is isomorphic to $\mathfrak{r}_{-1,-1}$, then $G$ admits only one left invariant anti-K\"{a}hler structure, up to equivalence.
  If $\mathfrak{g}$ is isomorphic to $\mathfrak{aff}(\mathbb{C})_{\mathbb{R}}$, then $G$ admits only a two-parameter family of non-equivalent anti-K\"{a}hler structures.
\end{theorem}

\begin{proof} $ $\newline

\noindent \textbf{Case 1.}
First, we fix the structure $\mu_{1,0,+1}$ obtained in (\ref{caso1a}). We have that $\mu_{1,0,+1}$ is equivalent to $\mu_{a,b,\varepsilon}$, for all $a,b\in \mathbb{R}$ ($a,b\neq0$) and $\varepsilon=\pm 1$. Indeed, we take the change of basis given by the matrix
\begin{eqnarray*}
\varphi &:=&
\frac{1}{2(a^2+b^2)}
\left[
\begin {array}{cccc}
 \varepsilon r a &\varepsilon s b & \varepsilon r b& -\varepsilon s a\\
-\varepsilon s b &\varepsilon r a & \varepsilon s a&  \varepsilon r b\\
            -r b &            s a &             r a&              s b\\
            -s a &           -r b &            -s b&              r a
\end {array} \right]
\end{eqnarray*}
whose inverse is
\begin{eqnarray*}
\varphi^{-1} &=&
\frac{1}{2(a^2+b^2)}
\left[
\begin {array}{cccc}
 \varepsilon r a & \varepsilon s b &- r b  &  s a\\
-\varepsilon s b & \varepsilon r a &- s a  &- r b\\
 \varepsilon r b &-\varepsilon s a &  r a  &  s b\\
 \varepsilon s a & \varepsilon r b & -s b  &  r a
\end {array} \right]
\end{eqnarray*}
where $r=a^2+b^2+1$ and $s=a^2+b^2-1$.

A straightforward computation shows that $\varphi$ preserves the inner product $\ip$
 and the complex structure $J$, and that $\varphi$  is an isomorphism between $\mu_{1,0,+1}$ and $\mu_{a,b,\varepsilon}$.

\medskip

\noindent \textbf{Case 2.}
We recall that, by Remark \ref{baseOrtonormal}, if we have an anti-Hermitian vector space  $(V, \ip, J)$ of real dimension $4$ we have an associated $2$-dimensional complex vector space  $(V,\mathbb{C})$ and a $\mathbb{C}$-symmetric inner product $\ipd$. And conversely, given a complex vector space endowed with a $\mathbb{C}$-symmetric inner product we have its associated  anti-Hermitian vector space.

Now, we consider the structure of real Lie algebra $\mu_{t_{1},t_{2},t_{3},t_{4}}$ defined in (\ref{caso2}) on the anti-Hermitian vector space
$(V, \ip, J)$, with $V=\operatorname{Span}_{\mathbb{R}}\{X,JX,Y,JY\}$. Since $J$ is a bi-invariant complex structure,
$\mu_{t_{1},t_{2},t_{3},t_{4}}$ induces a complex Lie bracket $\mu_{z_{1},z_{2}}$ on the $2$-dimensional complex vector space $(V,\mathbb{C})$
given by:
\begin{eqnarray*}
\mu_{z_{1},z_{2}}(X,Y)&=&z_{1}X + z_{2}Y,
\end{eqnarray*}
where $z_{1}=t_{1}+ \sqrt{-1} t_{2}$ and $z_{2}=t_{3}+ \sqrt{-1} t_{4}$. Conversely, if $(V,\mathbb{C})$ is endowed with a Lie bracket $\mu_{z_{1},z_{2}}$, where $z_{1},z_{2}$ are complex numbers ($z_1,z_2 \neq 0$), we have a member of  (\ref{caso2}), $\mu\,_{t_{1},t_{2},t_{3},t_{4}}$, where $t_{1}=\mathfrak{Re}(z_{1})$, $t_{2}=\mathfrak{Im}(z_{1})$, $t_{3}=\mathfrak{Re}(z_{2})$ and $t_{4}=\mathfrak{Im}(z_{2})$.

From Remark \ref{igualdadgrupos}, we have two Lie algebra structures $\mu_{t_{1},t_{2},t_{3},t_{4}}$ and $\mu_{s_{1},s_{2},s_{3},s_{4}}$ on $(V, \ip, J)$ are isomorphic via a linear map $\varphi$ that preserves the inner product $\ip$ and the complex structure $J$ if and only if the corresponding structures of complex Lie algebras $\mu_{z_{1},z_{2}}$ and $\mu_{w_{1},w_{2}}$ on $(V,\ipd )$ are isomorphic via a linear map $\widehat{\varphi}$ that preserves $\ipd$, where $z_{1}=t_{1}+ \sqrt{-1} t_{2}$, $z_{2}=t_{3}+ \sqrt{-1} t_{4}$, $w_{1}=s_{1}+ \sqrt{-1} s_{2}$ and $w_{2}=s_{3}+ \sqrt{-1} s_{4}$.

Thus, in this case,  we can study the equivalence of anti-Kh\"{a}ler structures  from the equivalence of structures of Lie algebras $\mu_{z_{1},z_{2}}$ on  $(V, \ipd)$ via isomorphisms that preserve the $\mathbb{C}$-symmetric inner product $\ipd$.

We fix $z_{1},z_{2}\in \mathbb{C}$ and take $\varphi \in O(V,\ipd) \cong O(2,\mathbb{C})$. One can verify that
\begin{equation}\label{equiv}
\varphi \cdot \mu_{z_{1},z_{2}}= \mu_{\epsilon w_{1}, \epsilon w_{2}},
\end{equation}
with $w_{1}X+w_{2}Y =\varphi(z_{1}X+z_{2}Y)$ and $\epsilon =\det \varphi = \pm 1$.
Thus, if $\mu_{z_{1},z_{2}}$ and $\mu_{w_{1}, w_{2}}$ are equivalent (isomorphic) then $z_{1}^2 + z_{2}^2 = w_{1}^2 + w_{2}^2$, since $\varphi\in O(V,\ipd)$ .

Now, we suppose that $z_{1},z_{2},w_{1},w_{2}$ are complex numbers such that $z_{1}^2 + z_{2}^2 = w_{1}^2 + w_{2}^2$. We want to prove that its corresponding structures of Lie algebras $\mu_{z_{1},z_{2}}$ and $\mu_{w_{1}, w_{2}}$ are in the same $O(V,\ipd)$-orbit. We have two possibilities:

\noindent \textbf{a.} If $z_{1}^2 + z_{2}^2\neq 0$,  we consider the  following orthogonal bases of $(V,\ipd)$
$$
\mathcal{B}=\{z_{1}X+z_{2}Y, -z_{2}X+z_{1}Y\} \hspace{0.5cm}\text{and} \hspace{0.5cm}
\mathcal{B}'=\{w_{1}X+w_{2}Y, -w_{2}X+w_{1}Y\}
$$
and we take the linear operator $\varphi$ of $V$ which sends $z_{1}X+z_{2}Y$ onto $w_{1}X+w_{2}Y$ and $-z_{2}X+z_{1}Y$ onto $ -w_{2}X+w_{1}Y$. A straightforward computation shows that $\varphi$ preserves square norms, so by polarization we obtain that $\varphi\in O(V,\ipd)$
and furthermore, by (\ref{equiv}), $\mu_{z_{1},z_{2}}$ and $\mu_{w_{1}, w_{2}}$ are equivalent.

\medskip

\noindent \textbf{b.} If $z_{1}^2 + z_{2}^2 = 0$, this implies that
\begin{eqnarray}
(z_{1}, z_{2}) &=& z(1, \sqrt{-1}) \mbox{ or } z(1, -\sqrt{-1}),
\end{eqnarray}
for some $z\in \mathbb{C}$, ($z\neq 0$).

Let us first prove that, given $z\in \mathbb{C}\setminus\{0\}$, there exists $\varphi_{z} \in O(V,\ipd)$ such that
$\varphi_{z} \cdot \mu_{1,\sqrt{-1}}= \mu_{z, z\sqrt{-1}}$. In fact, we consider the bases of the complex vector space $(V,\mathbb{C})$ given by
$$
\mathcal{B}=\{X+\sqrt{-1}Y, X-\sqrt{-1}Y\} \hspace{0.5cm}\text{and} \hspace{0.5cm}
\mathcal{B}'=\{zX+\sqrt{-1}zY, \tfrac{1}{z}X-\tfrac{1}{z}\sqrt{-1}Y\},
$$
and $\varphi_{z}$ the linear operator defined by $\varphi_{z}(X+\sqrt{-1}Y)=zX+z\sqrt{-1}Y$ and $\varphi_{z}(X-\sqrt{-1}Y)=\tfrac{1}{z}X-\tfrac{1}{z}\sqrt{-1}Y$.
It is easy to check that $\varphi_{z}$ is a linear operator preserving the $\mathbb{C}$-symmetric inner product $\ipd$.

Now, given $w\in \mathbb{C}\setminus\{0\}$ we have that $\mu_{1,-\sqrt{-1}}$ and  $\mu_{w, -w\sqrt{-1}}$ are equivalent via $\varphi_{\tfrac{1}{w}}$.

Besides, $\mu_{1,\sqrt{-1}}$ and  $\mu_{1, -\sqrt{-1}}$ are equivalent via $\varphi_{o} \in O(V,\ipd)$, the linear operator whose matrix with respect to the basis $\mathcal{B}=\{X+\sqrt{-1}Y, X-\sqrt{-1}Y\}$ is
\begin{eqnarray*}
[\varphi_{o}]_{\mathcal{B}}=\left[\begin{array}{c c}
0 & -1\\
1 & 0
\end{array}\right].
\end{eqnarray*}

Finally, if $z_{1}^2 + z_{2}^2 = w_{1}^2 + w_{2}^2=0$, we have that
$$
(z_{1}, z_{2})=
\begin{cases}
z(1, \sqrt{-1}) \,\,\,\text{or}\\
z(1, -\sqrt{-1}),
\end{cases}\hspace{0.5cm}\text{and}\hspace{0.5cm}
(w_{1}, w_{2})=
\begin{cases}
w(1, \sqrt{-1})\,\,\,\text{or}\\
w(1, -\sqrt{-1}),
\end{cases}
$$
for certain $z, w\in \mathbb{C} \setminus\{0\}$. Making the suitable compositions of the above linear operators,
we have that if $z_{1}^2 + z_{2}^2 = w_{1}^2 + w_{2}^2=0$, then $\mu_{z_1, z_2}$ and $\mu_{w_1,w_2}$ are isomorphic via $\widehat{\varphi} \in O(V,\ipd)$.
\end{proof}

\subsection{The curvature of the left invariant anti-K\"{a}hler structures in dimension $4$.}
$\mbox{ }$\newline
We consider the anti-K\"{a}hler structures in the family obtained in (\ref{caso1a}).
An easy computation shows that
\begin{eqnarray*}
  \begin{array}{ccc}
    \nabla_{X}X = -a JY,  & \nabla_{X}Y = a JX, & \nabla_{Y}Y = -\varepsilon b JX.
  \end{array}
\end{eqnarray*}
By using that $\nabla$ is torsion free and $\nabla J \equiv 0$, we can compute the remaining
values $\nabla_U V$ with $U,V \in \mathcal{B}:=\{X,JX,Y,JY\}$ and show that
\begin{eqnarray*}
  \nabla_U \nabla_V W & = & \nabla_V \nabla_U W,\, \forall\, U,V,W \in \mathcal{B}.
\end{eqnarray*}
Furthermore, it is fairly easy to show that $\nabla_{[U,V]}W =0$ for all $U,V,W$ in $\mathcal{B}$.
Then, anti-K\"{a}hler structures in the family are flat.

For the anti-K\"{a}hler structures obtained in (\ref{caso2}), we have
\begin{eqnarray*}
  \begin{array}{ccc}
    \nabla_{X}X = -t_1Y - t_{2}JY,  & \nabla_{X}Y = t_1X + t_2JX, & \nabla_{Y}Y = t_3X+t_4JX,
  \end{array}
\end{eqnarray*}
and
\begin{eqnarray*}
  \nabla_U \nabla_V W & = & \nabla_V \nabla_U W,\, \forall\, U,V,W \in \mathcal{B}.
\end{eqnarray*}
Furthermore, we have
\begin{eqnarray*}
  R(X,Y) &=& \left[\begin{array}{rr} 0 &  H\\ -H  & 0\end{array}\right]
\end{eqnarray*}
where
\begin{eqnarray*}
H & = & \left[\begin{array}{rr}
\mathfrak{Re}(\zeta) &  -\mathfrak{Im}(\zeta) \\
\mathfrak{Im}(\zeta) &   \mathfrak{Re}(\zeta) \end{array}\right],
\end{eqnarray*}
here $\zeta = \ipda{z_1X+z_2Y}{z_1X+z_2Y} = z_1^2 + z_2^2$, $z_1:=t_1+t_2\sqrt{-1}$ and $z_2:= t_3+t_4\sqrt{-1}$.
And so, the anti-K\"{a}hler structure is flat if and only if $\zeta = 0$.
Regarding the Ricci curvature tensor, we have that the Ricci operator $\operatorname{Ric}$, $\operatorname{Rc}(\cdot,\cdot)=g(\operatorname{Ric}\cdot,\cdot)$,
is given by the matrix
\begin{eqnarray}
  \operatorname{Ric} & = & -2\left[\begin{array}{rr} H &  0 \\ 0  & H\end{array}\right]
\end{eqnarray}
with respect to the basis $\mathcal{B}$. Hence, the anti-K\"{a}hler structure is Einstein
if and only if $\mathfrak{Im}(\zeta) = 0$, and in that case, the cosmological
constant is $-2\zeta$. Moreover, the anti-K\"{a}hler structure is Ricci flat
if and only if it is flat.



\begin{thebibliography}{BML}



\bibitem[1]{AndradaBarberisDotti1}
\textsc{Andrada A., Barberis M. L. and Dotti I.}:
\emph{Classification of abelian complex structures on 6-dimensional Lie algebras}.
{Journal of the London Mathematical Society}
\textbf{83} \textnumero{1}, pp.{232--255}  (2011)

\bibitem[2]{AndradaBarberisDotti2}
\textsc{Andrada A., Barberis M. L. and Dotti I.}:
\emph{Corrigendum: Classification of abelian complex structures on 6-dimensional Lie algebras}.
{Journal of the London Mathematical Society}
\textbf{87} \textnumero{2}, pp.{319--320}  (2013)

\bibitem[3]{BarberisDottiMiatello}
\textsc{Barberis M. L., Dotti I. and Miatello R.}:
\emph{On certain locally homogeneous Clifford manifolds}.
{Annals of Global Analysis and Geometry}
\textbf{13} \textnumero{3}, pp.{289--301}  (1995)




\bibitem[4]{BorowiecFerrarisFrancavigliaVolovich}
\textsc{Borowiec A., Ferraris M., Francaviglia M. and Volovich I.}: 
\emph{Almost-complex and almost-product Einstein manifolds from a variational principle}.
{Journal of Mathematical Physics}
\textbf{40} \textnumero{7}, pp.{3446--3464}  (1999)

\bibitem[5]{BorowiecFrancavigliaVolovich}
\textsc{Borowiec A., Francaviglia M. and Volovich I.}:
\emph{Anti-Kählerian manifolds.}.
{Differential Geometry and its Applications}
\textbf{12} \textnumero{3}, pp.{281--289}  (2000)

\bibitem[6]{CastroHervellaGarciaRio}
\textsc{Castro R., Hervella L.M. and Garc\'{i}a-Rio E.}:
\emph{Some examples of almost complex manifolds with Norden metric}.
{Rivista di Matematica della Università di Parma}
Serie 4 Volume \textbf{15}, pp.{133--141}  (1989)

\bibitem[7]{DegirmenciKarapazar1}
\textsc{De{\u{g}}irmenci N. and Karapazar {\c{S}}}:
\emph{Spinors on K\"{a}hler-Norden manifolds}.
{Journal of Nonlinear Mathematical Physics }
\textbf{17} \textnumero{1}, pp.{27--34}  (2010)

\bibitem[8]{DegirmenciKarapazar2}
\textsc{De{\u{g}}irmenci N. and Karapazar {\c{S}}}:
\emph{Schr\"{o}dinger--Lichnerowicz like formula on K\"{a}hler-Norden manifolds}.
{International Journal of Geometric Methods in Modern Physics}
\textbf{9} \textnumero{1},  1250010 (14 pages)  (2012)

\bibitem[9]{GanchevBorisov}
\textsc{Ganchev G.T. and Borisov A.W.}:
\emph{Note on the almost complex manifolds with Norden Metric}.
{Comptes rendus de l'Académie Bulgare des Sciences.}
\textbf{39} \textnumero{5}, pp.31--34  (1986)

\bibitem[10]{GribachevMekerovDjelepov}
\textsc{Gribachev K. I., Mekerov D.G. and Djelepov G.D.}:
\emph{Generalized B-manifold}.
{Comptes rendus de l'Académie Bulgare des Sciences}
\textbf{38} \textnumero{3 }, pp.{299--302}  (1985)


\bibitem[11]{Harvey}
\textsc{Harvey F. R.}:
\emph{Spinors and Calibrations}.
{Perspectives in Mathematics}
\textbf{9} first ed. Academic Press, Boston (1990)

\bibitem[12]{IscanSalimov}
\textsc{\.{I}{\c{s}}can M. and Salimov A. A.}:
\emph{On K\"{a}hler-Norden manifolds}.
{Proceedings - Mathematical Sciences}
\textbf{119} \textnumero{1}, pp.{71--80}  (2009)



\bibitem[13]{Manev1}
\textsc{Manev M.}:
\emph{Classes of real isotropic hypersurfaces of a Kaehler manifold with B-metric}.
{Comptes rendus de l'Académie Bulgare des Sciences.}
\textbf{55} \textnumero{4}, pp.27--32  (2002)

\bibitem[14]{Mekerov2}
\textsc{Mekerov D.}:
\emph{Connection with parallel totally skew-symmetric torsion on almost complex manifolds with Norden metric}.
{Comptes rendus de l'Académie Bulgare des Sciences}
\textbf{62} \textnumero{12}, pp.{1501--1508}  (2009)

\bibitem[15]{Mekerov3}
\textsc{Mekerov D.}:
\emph{On the geometry of the connection with totally skew-symmetric torsion on almost complex manifolds with Norden metric}.
{Comptes rendus de l'Académie Bulgare des Sciences.}
\textbf{63} \textnumero{1}, pp.{19--28}  (2010)

\bibitem[16]{Nannicini}
\textsc{Nannicini A.}:
\emph{Generalized geometry of Norden manifolds}.
{Journal of Geometry and Physics}
\textbf{99}, pp.{244--255}  (2016)

\bibitem[17]{Olszakk}
\textsc{Olszak K. (f.k.a. S{\l}uka K.)}: 
\emph{On the Bochner conformal curvature of K\"{a}hler-Norden manifolds}.
{Central European Journal of Mathematics}
\textbf{3} \textnumero{2}, pp.{309--317}  (2005)

\bibitem[18]{Sluka}
\textsc{S{\l}uka K.}:
\emph{On K\"{a}hler manifolds with Norden metrics}.
{Analele {\c S}tiin{\c t}ifice ale Universit{\u{a}}{\c t}ii \quotedblbase{Alexandru Ioan Cuza}{\textquotedblleft}{ }din Ia{\c s}i. Matematic{\u{a}}}.
Tomul XLVII f.1, pp.{105--122}  (2001)


\bibitem[19]{Sluka2}
\textsc{S{\l}uka K.}:
\emph{On the curvature of K\"{a}hler--Norden manifolds}.
{Journal of Geometry and Physics}.
\textbf{54} \textnumero{2}, pp. {131--145} (2005)

\bibitem[20]{Teofilova5}
\textsc{Teofilova M.}:
\emph{Lie groups as Kähler manifolds with Killing Norden metric}.
{Comptes rendus de l'Académie Bulgare des Sciences}
\textbf{65} \textnumero{6}, pp.{733--742}  (2012)

\bibitem[21]{Wang}
\textsc{Wang H. C.}:
\emph{Complex parallelisable manifolds}.
{Proceedings of the American Mathematical Society}
\textbf{5} \textnumero{5}, pp.{771--776}  (1954)
































\end{thebibliography}
\end{document}